\documentclass{icmart}
\usepackage{color}

\newtheorem{theorem}{Theorem}[section]

\newtheorem{lemma}[theorem]{Lemma}

\newtheorem{conjecture}[theorem]{Conjecture}

\theoremstyle{definition}
\newtheorem{definition}[theorem]{Definition}

\newcommand{\tr}{\mathrm{Tr}}
\renewcommand{\v}{r}
\renewcommand{\L}{\mathcal{L}}
\newcommand{\E}{\mathbb{E}}

\newcommand{\R}{\mathbb{R}}

\newcommand{\norm}[1]{\left\| #1 \right\|}

\newcommand{\setof}[1]{\left\{ #1 \right\}}

\newcommand{\expec}[2]{{\operatorname{\mathbb{E}}\displaylimits}_{#1}  #2 }

\newcommand\trace[1]{\mathrm{Tr} \left[#1 \right]}

\newcommand{\mydet}[1]{\det\left(#1\right)}
\newcommand{\charp}[2]{\chi \left[#1 \right] \left(#2 \right)}
\newcommand{\mixed}[2]{\mu \left[#1 \right] \left(#2 \right)}

\def\ceil#1{\left\lceil #1 \right\rceil}
\def\sizeof#1{\left|#1  \right|}

\newcommand{\imag}[1]{\mathsf{Im} (#1)}
\newcommand{\Complex}[1]{\mathbb{C}^{#1}}
\newcommand{\C}{\mathbb{C}}
\newcommand{\Reals}[1]{\mathbb{R}^{#1}}
\newcommand{\smin}{\mathrm{smin}}
\newcommand{\smaxx}{\mathrm{smax}}
\newcommand{\smax}{\overrightarrow{\mathrm{smax}}}

\def\pleq{\preceq}

\def\Branden{Br\"{a}nd\'{e}n}

\theoremstyle{definition}

\contact[adam.marcus@yale.edu]{Yale University and Crisply, Inc}
\contact[spielman@cs.yale.edu]{2. Yale University}
\contact[niksri@microsoft.com]{3. Microsoft Research, India}

\title[Ramanujan Graphs and the Solution
of the Kadison--Singer Problem]{Ramanujan Graphs and the Solution
of the Kadison--Singer Problem}

\author[Adam W. Marcus, Daniel A. Spielman, Nikhil Srivastava]{
Adam W. Marcus\thanks{Research partially supported by an NSF Mathematical Sciences Postdoctoral Research Fellowship, Grant No. DMS-0902962.}
\qquad
Daniel A. Spielman\thanks{Research partially supported by NSF grants CCF-0915487 and  CCF-1111257, a Simons Investigator Award, and a MacArthur Fellowship.}
\qquad
Nikhil Srivastava
}

\begin{document}

%

\maketitle

\begin{abstract}
We survey the techniques used in our recent resolution of the  Kadison--Singer problem 
  and proof of existence of Ramanujan Graphs of every degree: 
  mixed characteristic polynomials and 
  the method of interlacing families of polynomials.
To demonstrate the method of interlacing families of polynomials, we give a simple proof of 
  Bourgain and Tzafriri's restricted invertibility principle in the isotropic case.
\end{abstract}



\section{Introduction}\label{sec:outline}
In a recent pair of papers \cite{IF1,IF2}, we prove the existence of infinite families of bipartite Ramanujan graphs of every degree and we
  affirmatively resolve the Kadison--Singer Problem.
The techniques that we use in the papers are closely related.
In both we must show that certain families of matrices contain particular matrices of small norm.
In both cases, we prove this through a new technique that we call the \textit{method of interlacing families of polynomials}.
In the present survey, we review this technique and the polynomials that we analyze with it, the
  \textit{mixed characteristic polynomials}.

We begin by defining Ramanujan Graphs, explaining the Kadison--Singer Problem,
  and explaining how these problems are related.
In particular, we connect the two by demonstrating how they are both related to the
  problem of sparsifying graphs.

\subsection{Ramanujan Graphs}\label{ssec:raman}
Let $G$ be an undirected graph with vertex set $V$ and edge set $E$.
The {\em adjacency matrix} of $G$ is the symmetric matrix $A$ whose rows and columns are indexed
  by vertices in $V$ with entries
\[
  A (a,b) = \begin{cases}
1 & \text{if $(a,b) \in E$}
\\
0  & \text{otherwise}.
\end{cases}
\]
Since $A$ is symmetric it has $|V|$ real eigenvalues, which we will also refer to as the {\em eigenvalues of} $G$.

Consider a function $f:V\rightarrow \mathbb{R}$.  
Multiplication by $A$ corresponds to
  the operator that replaces the value of $f$ at a given vertex with the
  sum of the values at its neighbors in $G$.  
In this way, $A$ is related to random walks and diffusion on $G$.
It is well known that the speed of the convergence of these processes
  is determined by the eigenvalues of $A$ and related matrices.

We will restrict our attention to graphs that are connected and $d$-regular.
When $|V|$ is finite, it is easy to check that every such graph has an
eigenvalue of $d$ corresponding to the eigenvector of all $1$'s.
Furthermore, in the case that $G$ is bipartite,
one can check that the eigenvalues of $A$ are
symmetric about the origin.  Thus every finite bipartite $d$-regular
graph must also have an eigenvalue of $-d$.  Because these
eigenvalues are unavoidable (they are an artifact of being finite),
they are often referred to as the {\em trivial eigenvalues}.

The graphs on which random walks mix the fastest are those 
  whose non-trivial eigenvalues are as small as possible.
An infinite family of connected $d$-regular graphs
  all of whose non-trivial eigenvalues are at most $\alpha$
  for some constant $\alpha < d$
  is called a family of \textit{expander graphs}.
Constructing $d$-regular expanders with a small
  number of vertices (relative to $d$) is easy: for example, the
  complete graph on $d+1$ vertices has all non-trivial eigenvalues $-1$
  and the complete bipartite graph with $2d$ vertices has all
  non-trivial eigenvalues $0$.  
The interesting problem is to construct
  $d$-regular expanders with an arbitrarily large number of
  vertices.
Margulis \cite{Margulis73} was the first to find an explicit construction
  of such an infinite family.

Expander graphs have proved to be incredibly useful in a variety of contexts.
We refer the reader who is interested in learning more about expander
  graphs, with a focus on their applications in computer science, to the
  survey of Hoory, Linial, and Wigderson~\cite{ExpanderSurvey}.  
Many applications of expanders depend upon the magnitudes of their non-trivial eigenvalues.
A theorem of Alon and Boppana provides a bound on how small the non-trivial
  eigenvalues can be.

\begin{theorem}[\cite{alon86,nilli}]\label{thm:AlonBoppana}
For every integer $d \geq 3$ and every $\epsilon > 0$, there exists an $n_{0}$ so that 
  every $d$-regular graph $G$ with more than $n_{0}$ vertices has a non-trivial eigenvalue that is greater than $2 \sqrt{d-1} - \epsilon$.
\end{theorem}

The number $2\sqrt{d-1}$ in Theorem~\ref{thm:AlonBoppana} has a
  meaning: it is the spectral radius of the infinite $d-$regular tree, whose
  spectrum is the closed interval $[-2\sqrt{d-1}, 2\sqrt{d-1}]$
   (it has no trivial eigenvalues
  because it is not finite) \cite{ExpanderSurvey}.
Since Theorem~\ref{thm:AlonBoppana} says that no infinite family of
  $d$-regular graphs can have eigenvalues that are asymptotically
  smaller than $2\sqrt{d-1}$, we may view this infinite tree as being the
  ``ideal'' expander.
A natural question is whether there exist infinite families of finite $d$-regular
  graphs whose eigenvalues are actually as small as those of the tree.

Lubotzky, Phillips and Sarnak~\cite{LPS}
  and Margulis~\cite{Margulis} were the first to construct infinite
  families of such graphs.  Their constructions were Cayley graphs, and
  they exploited the algebraic properties of the underlying groups to prove
  that all of the nontrivial eigenvalues of their graphs have absolute value at 
  most $2\sqrt{d-1}$.
Their proofs required the proof of the Ramanujan Conjecture, 
 and so they named the graphs they obtained
  {\em Ramanujan graphs}.  
As of 2013, all known infinite families of Ramanujan graphs were obtained
   via constructions similar to
  \cite{LPS, Margulis}.
As a result, all known families of Ramanujan graphs had degree $p^k + 1$ for $p$ a prime and
  $k$ a positive integer.

The main theorem of \cite{IF1} is that there exist infinite families of $d$-regular 
  bipartite Ramanujan graphs for every integer $d \geq 3$.
This is achieved by proving a variant of a conjecture of Bilu and Linial~\cite{BiluLinial},
  which implies that every $d-$regular Ramanujan graph 
  has a $2-$cover which is also Ramanujan, immediately establishing the existence of an
  infinite sequence.
In contrast to previous results, the proof is completely elementary, and we will
  sketch most of it in this survey.

Bilu and Linial's conjecture is a purely linear algebraic statement about 
  {\em signings} of adjacency matrices. 
To define a signing, recall that we can write the adjacency matrix of any graph
  $G=(V,E)$ as 
$$ A = \sum_{(a,b)\in E} A_{(a,b)},$$
where $A_{(a,b)}$ is the adjacency matrix of a single edge $(a,b)$.
Then, a signing is any matrix of the form
  $$\sum_{(a,b)\in E} s_{(a,b)}A_{(a,b)},$$
where $s_{(a,b)}\in \{-1,+1\}$ are signs.
A graph with $m$ edges has exactly $2^m$ signings.

Bilu and Linial conjectured that every $d-$regular adjacency matrix $A$ has a
  signing $A_s$ with $\|A_s\|\le 2\sqrt{d-1}$.
We prove the following weaker statement, which is equivalent to their conjecture
  in the bipartite case, as in this case the eigenvalues are symmetric about zero.


\begin{theorem}\label{thm:IF1}
Every $d-$regular adjacency matrix $A$ has a signing $A_s$ with
$$\lambda_{max}(A_s)\le 2\sqrt{d-1}.$$
\end{theorem}
This is a statement about the existence of a certain sum of rank two matrices of type $s_{(a,b)}A_{(a,b)}$,
  but it it useful to rewrite it as a statement about a sum of rank one matrices by making the substitution
$$ s_{(a,b)}A_{(a,b)}= (e_a+ s_{(a,b)}e_b)(e_a+s_{(a,b)}e_b)^T-e_ae_a^T-e_be_b^T,$$
where $e_a$ is the standard basis vector with a $1$ in position $a$.
For a $d-$regular graph, we now have
\begin{equation}\label{eqn:signedlaplacian} A_s=\sum_{(a,b)\in E} s_{(a,b)}A_{(a,b)}= 
  \sum_{(a,b)\in E} (e_a + s_{(a,b)}e_b)(e_a+s_{(a,b)}e_b)^T 
  - dI.\end{equation}
So, Theorem~\ref{thm:IF1} is equivalent to the statement that there
  is a choice of $s_{(a,b)}$ for which 
$$\lambda_{max}\left(\sum_{(a,b)\in E} (e_a + s_{(a,b)}e_b)(e_a+s_{(a,b)}e_b)^T
\right)\le d+2\sqrt{d-1}.$$
The existence of such a choice can be written in probabilistic terms by defining
  for each $(a,b)\in E$ a random vector 
\begin{equation}\label{eqn:rabdef} r_{(a,b)} :=
\begin{cases}
  (e_a+e_b) &          \textrm{with probability $1/2$ and}
\\
 (e_a-e_b)  & \textrm{with probability
  $1/2$}.
\end{cases}
\end{equation}
Notice that 
\begin{equation}\label{eqn:signisotropic}\E \sum_{(a,b)\in E}
r_{(a,b)}r_{(a,b)}^T=dI.\end{equation}
Thus, Theorem \ref{thm:IF1} is equivalent to the statement that for every
  $d-$regular $G=(V,E)$, 
\begin{equation}\label{eqn:IF1}\lambda_{max}\left(\sum_{(a,b)\in E} r_{(a,b)}r_{(a,b)}^T\right)\le
  \lambda_{max}\left(\E \sum_{(a,b)\in E} r_{(a,b)}r_{(a,b)}^T\right) +2\sqrt{d-1}\end{equation}
  with positive probability.

Such a sum may be analyzed using tools of random matrix theory, but this
  approach does not give the sharp bound we require, and it is known that it
  cannot in general as there are graphs for which the desired signing
  is exponentially rare (consider a union of disjoint cliques on $d$ vertices).

The main subject of this survey is an approach that succeeds in
  proving \eqref{eqn:IF1} exactly.
The methodology also succeeds in resolving several other important questions
  about sums of independent random rank one matrices, including Weaver's
  conjecture and thereby the Kadison--Singer problem.
We review these first and describe their connection to Ramanujan graphs before
  proceeding to describe the actual technique.
The proof of \eqref{eqn:IF1} and Theorem \ref{thm:IF1} will be sketched in Section
  \ref{sec:matchingpoly}.

\subsection{Sparse Approximations of Graphs}\label{ssec:approx}
Spielman and Teng~\cite{SpielmanTengSparse} observed that one can view an expander graph
  as an approximation of a complete graph, and 
  asked if one could find analogous approximations of arbitrary graphs.
In this context, it is more natural to consider the class of general weighted graphs rather than
  just unweighted $d-$regular graphs, and to study the 
  {\em Laplacian matrix} instead of the adjacency matrix.
Recall that the Laplacian of a weighted graph $G=(V,E,w)$ may be defined as the following sum of
  rank one matrices over the edges:
\[
  L_{G} = \sum_{(a,b) \in E} w_{(a,b)} (e_{a} - e_{b}) (e_{a} - e_{b})^{T}.
\]
In the unweighted $d-$regular case, it is easy to see that $L=dI-A$, so the
eigenvalues of the Laplacian are just $d$ minus the eigenvalues of the adjacency
  matrix.
The Laplacian matrix of a graph always has an eigenvalue of $0$; this is a 
  trivial eigenvalue, and the corresponding eigenvectors are the constant vectors.

Following Spielman and Teng, we say that two graphs $G$ and $H$ on the same
  vertex set $V$ are {\em spectral approximations} of each other if their
  Laplacian quadratic forms multiplicatively approximate each other:
  $$ \kappa_1\cdot x^TL_Hx\leq x^TL_Gx\leq \kappa_2 \cdot x^TL_Hx\qquad\forall x\in\R^V,$$
for some approximation factors $\kappa_1,\kappa_2>0$. We will write this as
	$$ \kappa_1\cdot L_H \pleq L_G\pleq \kappa_2\cdot L_H,$$
where $A\pleq B$ means that $B-A$ is positive semidefinite, i.e., $x^T(B-A)x\ge
  0$ for every $x$.

The complete graph on $n$ vertices, $K_{n}$, is the graph with an edge of weight $1$ between every pair of vertices.
All of the eigenvalues of $L_{K_{n}}$ other than $0$ are equal to $n$.
If $G$ is a $d$-regular non-bipartite Ramanujan graph, then $0$ is the trivial 
  eigenvalue of its Laplacian matrix, $L_{G}$,
  and all of the other eigenvalues of $L_{G}$ are between $d - 2 \sqrt{d-1}$ and $d + 2 \sqrt{d-1}$.
After a simple rescaling, this allows us to conclude that
\[
  (1 - 2 \sqrt{d-1}/d) L_{K_{n}} 
  \pleq (n/d) L_{G} 
  \pleq  (1 + 2 \sqrt{d-1}/d)  L_{K_{n}} .
\]
So, $(n/d) L_{G}$ is a good approximation of $L_{K_{n}}$.

Batson, Spielman and Srivastava proved that every weighted graph has an approximation that is almost this good.
\begin{theorem}[\cite{BSS}]\label{thm:BSSgraphs}
For every $d>1$
  and every weighted graph $G=(V,E,w)$ on $n$ vertices,
  there exists a weighted graph $H=(V,F,\tilde{w})$ with $\ceil{d(n-1)}$
  edges that satisfies:
\begin{equation}\label{eqn:BSSgraph}
\left(1-\frac{1}{\sqrt{d}}\right)^2 L_G  \pleq L_{H} \pleq
\left(1+\frac{1}{\sqrt{d}}\right)^2 L_{G}.
\end{equation}
\end{theorem}

However, their proof had very little to do with graphs.
In fact, they derived their result from the following theorem about sparse weighted
  approximations of sums of rank one matrices.
\begin{theorem}[\cite{BSS}]\label{thm:BSSmatrices}
Let $v_{1}, v_{2}, \dots , v_{m}$ be vectors in $\Reals{n}$ with
\[
  \sum_{i} v_{i} v_{i}^{T} = V.
\]
For every $\epsilon \in (0,1)$,
  there exist non-negative real numbers
  $s_{i}$
  with
\[
\sizeof{\setof{i : s_{i} \not = 0}}
\leq 
\ceil{n / \epsilon^{2}}
\]
so that
\begin{equation}\label{eqn:BSSvecs}
(1-\epsilon)^{2} V 
  \pleq \sum_{i} s_{i} v_{i} v_{i}^{T}
  \pleq (1+\epsilon)^{2} V.
\end{equation}
\end{theorem}
Taking $V$ to be a Laplacian matrix written as a sum of outer products and
  setting $\epsilon=1/\sqrt{d}$ immediately yields Theorem \ref{thm:BSSgraphs}.

Theorem \ref{thm:BSSmatrices} is very general and turned out to be useful in a
  variety of areas including graph theory, numerical linear algebra, and metric
  geometry (see, for instance, the survey of Naor \cite{Naor}).
One of its limitations is that it provides no guarantees on the weights $s_i$
  that it produces, which can vary wildly.  
So it is natural to ask: is there a version of Theorem \ref{thm:BSSmatrices} in
  which all the weights are the same?

This may seem like a minor technical point, but it is actually a fundamental
  difference.
In particular, Gil Kalai observed that the statement of Theorem~\ref{thm:BSSmatrices} with
  $V=I$ is similar to Weaver's Conjecture, which was known to imply a positive 
  solution to the Kadison--Singer Problem.
It turns out that the natural unweighted variant of it is essentially {\em the
  same} as Weaver's conjecture.
We discuss the Kadison--Singer problem and this connection in the next
   section.
\subsection{The Kadison-Singer Problem and Weaver's Conjecture}\label{ssec:ksWeaver}
In 1959, Kadison and Singer \cite{orig_ks} asked the following fundamental question: does
  every pure state on the abelian von Neumann algebra $D(\ell_2)$ of diagonal operators
  on $\ell_2$ have a unique extension to a pure state on $B(\ell_2)$, the von
  Neumann algebra of all bounded operators on $\ell_2$?
In their original paper, they suggested an approach to resolving this question:
  they showed that the answer is yes if every operator in $B(\ell_2)$ can be
  `paved' by a constant number of operators which are strictly smaller in the operator norm.
Beginning with the work of Anderson~\cite{anderson1979a,anderson1979b,anderson1981}, this was shown to be equivalent to several
  combinatorial questions about decomposing {\em finite} matrices into a small
  number of strictly smaller pieces.

Among these questions is Akemann and Anderson's ``projection paving
  conjecture''~\cite{akemann1991lyapunov}, which Nik Weaver 
  \cite{weaver}
  later showed was
  equivalent to the following  discrepancy-theoretic conjecture that he called
  $KS_{2}$.
\begin{conjecture}\label{conj:weaver}
There exist positive constants $\alpha$ and $\epsilon$ so that for every 
  $n$ and $d$ and every 
  set
  of vectors $v_{1}, \dots , v_{n} \in \mathbb{C}^{d}$ such
  that $\norm{v_{i}} \leq \alpha$ for all $i$ and
\[
  \sum_{i} v_{i} v_{i}^{*} = I,
\]
there exists a partition of $\setof{1, \dots , n}$ into two sets $S_{1}$ and $S_{2}$
  so that for each $j \in \setof{1,2}$
\begin{equation}\label{eqn:weaver}
  \Big\| \sum_{i \in S_{j}} v_{i} v_{i}^{*} \Big\| < 1-\epsilon.
\end{equation}
\end{conjecture}

To see the similarity between this conjecture and Theorem \ref{thm:BSSmatrices}, observe that for any
partition $S_1\cup S_2$:
\[
   \sum_{i \in S_{1}} v_{i} v_{i}^{*} + \sum_{i \in S_{2}} v_{i} v_{i}^{*}  = I,
\]
so that condition \eqref{eqn:weaver}
  is equivalent to
\[
  \epsilon I \pleq  {\sum_{i \in S_{1}} v_{i} v_{i}^{*}} \pleq (1-\epsilon) I.
\]
Thus, choosing a subset of the weights $s_i$ to be non-zero in Theorem
  \ref{thm:BSSmatrices} is similar to choosing the set $S_{1}$.
The difference is that Conjecture~\ref{conj:weaver} assumes a bound on
the lengths of the vectors $v_i$ and in return requires the stronger conclusion
  that all of the $s_i$ are either $0$ or $1$.
It is easy to see that long vectors are an obstacle to the existence of a good
  partition; an extreme example is provided by considering an orthonormal basis
  $e_1,\ldots,e_n$.
Weaver's conjecture asserts that this is the only obstacle.

Overcoming this seemingly small difference turns out to require substantial new machinery 
  beyond the techniques used in the proof of Theorem~\ref{thm:BSSmatrices}.
However, much of this machinery is built on two key ideas which are contained in \cite{BSS}.
The first is the use of ``barrier functions'' to bound the roots of polynomials,
  which is discussed in Section \ref{sec:unibarrier}.
The second, which was presented purely for motivational purposes in \cite{BSS}, is 
  the examination of expected characteristic polynomials.

As in the case of Ramanujan graphs, Weaver's conjecture can be written in terms of sums of
  independent random rank one matrices.
Given vectors $v_1,\ldots,v_m\in\Complex{d}$, define for each $i$ the random vector
  $r_i\in\Complex{2d}$ 
\begin{equation}\label{eqn:weaverpartition}
  r_{i} 
= \begin{pmatrix}
  v_{i} \\
  0_{d}
\end{pmatrix}
\quad
\textrm{ with probability $1/2$}
\
\quad 
\text{and}
\quad 
\begin{pmatrix}
  0_{d}\\
  v_{i} 
\end{pmatrix}
\quad
\textrm{ with probability $1/2$},
\end{equation}
where $0_d\in\Complex{d}$ is the zero vector.
Then it is easy to see that every realization of $r_1,\ldots,r_m$ corresponds to
  a partition $S_1\cup S_2=[m]$ in the natural way, and that
\[
\sum_{i} r_{i} r_{i}^{*}
=
\begin{pmatrix}
\sum_{i \in S_{1}} v_{i} v_{i}^{*} & 0 \\
0 & \sum_{i \in S_{2}} v_{i} v_{i}^{*}
\end{pmatrix}.
\]
Moreover, the norm of this matrix is the maximum of the norms of the matrices
  in the upper-left and lower-right blocks.
Thus, Weaver's conjecture is equivalent to the statement that when
 the  $\|v_i\|\le\alpha$, the following holds with positive probability:
\begin{equation}\label{eqn:weaverr} \lambda_{max}\left(\sum_{i=1}^m
r_ir_i^*\right)\le 1-\epsilon\end{equation}

Once again, it is possible to apply tools of random matrix theory to analyze
  this sum.
This gives a proof of the conjecture with $\alpha=1/\log n$, essentially
  recovering a result of Bourgain and Tzafriri \cite{btks}, which was
  essentially the best partial solution to Kadison--Singer until recently.

The main result of \cite{IF2} is the following strong form of Weaver's
conjecture.
\begin{theorem}\label{thm:mss2main}
Let $v_{1}, \dots , v_{m}\in\Complex{d}$ satisfy
  $\sum_{i} v_{i} v_{i}^{*} = I$ and
  $\norm{v_{i}}^{2} \leq \alpha $ for all $i$.
Then, there exists a partition of $\setof{1,\dots ,m}$ into sets $S_{1}$
  and $S_{2}$ so that for $j \in \setof{1,2}$,
\begin{equation}
\norm{ \sum_{i \in S_{j}} v_{i} v_{i}^{*}}
\leq 
\frac{(1+\sqrt{2 \alpha})^{2}}{2}.
\end{equation}
\end{theorem}
We will sketch the proof of Theorem \ref{thm:mss2main}, which is closely related
  to the proof of Theorem \ref{thm:IF1},  in Sections \ref{sec:mixed} and \ref{sec:multibarrier}.

\subsection{Sums of Independent Rank One Random Matrices}
As witnessed by equations \eqref{eqn:signedlaplacian} and \eqref{eqn:weaverr}, the common thread in the 
  problems described above is that they can all be resolved by showing that a
  certain sum of independent random rank one matrices has small eigenvalues with
  nonzero probability.
Prior to this line of work, there were already well-developed tools in random
  matrix theory for reasoning about such sums,
  generally called Matrix Chernoff Bounds~\cite{AhlswedeWinter,rudelson2007sampling,tropp2012}.
As mentioned earlier, these provide bounds that are worse than those we require by a factor that is logarithmic
  in the dimension.
However, they  hold with high probability rather than the merely positive probability that we obtain.

Our approach to analyzing the
  eigenvalues of sums of independent rank one random matrices
  rests on the following connection between possible values of
  any particular eigenvalue, and the corresponding root of its
  expected characteristic polynomial.
We will use $\lambda_1\ge\lambda_2,\ldots,\ge\lambda_n\in\R$ to denote the
  eigenvalues of a Hermitian matrix as well as the roots of a real-rooted
  polynomial.
\begin{theorem}[Comparison with Expected Polynomial] \label{thm:comparison} Suppose
$r_1,\ldots,r_m\in\C^n$ are independent random vectors. Then, for every $k$,
$$ \lambda_k\left(\sum_{i=1}^m
r_ir_i^*\right)\le\lambda_k\left(\E\charp{\sum_{i=1}^m r_i
r_i^*}{x}\right),$$
with positive probability, and the same is true with $\ge$ instead of $\le$.\end{theorem}

In the special case when the $r_i$ are identically distributed with $\E
  r_i r_i^*=I$, there is short proof of Theorem \ref{thm:comparison} 
  that only requires univariate interlacing.
We present this proof as Lemma \ref{lem:induct} and  and use it to establish a
  variant of Bourgain and Tzafriri's restricted invertibility theorem.
In Section \ref{sec:mixed} we prove the theorem in full generality using
  tools from the theory of real stable polynomials. 
This yields mixed characteristic polynomials, which are then analyzed in
  Sections \ref{sec:matchingpoly} and \ref{sec:multibarrier} to prove 
  the existence of infinite families of Bipartite Ramanujan Graphs as well as
  Weaver's Conjecture

\section{Interlacing Polynomials}
A defining characteristic of the proofs in \cite{IF1} and \cite{IF2} 
  is that they analyze matrices solely through their characteristic polynomials.
This is perhaps a counterintuitive way to proceed; on the surface, we are losing information 
  by considering characteristic polynomials, which only know about eigenvalues and not
  eigenvectors.  
However, the structure we gain far outweighs the losses in two ways: the
  characteristic polynomials satisfy a number of algebraic
  identities which make calculating their averages tractable, and they are
  amenable to a set of analytic tools that do not naturally apply
  to matrices.

As hinted at earlier, we study the roots of 
   {\em averages} of polynomials. 
In general, averaging polynomials coefficient-wise can do unpredictable things to the roots.
For instance, the average of $(x-1)(x-2)$ and $(x-3)(x-4)$, which are both
  real-rooted quadratics, is $x^2-5x+7$, which has complex roots
  $2.5\pm\sqrt{3}i$.
Even when the roots of the average are real, there is in general no simple
  relationship between the roots of two polynomials and the roots of their
  average.

The main insight is that there are nonetheless many situations where averaging the
  coefficients of polynomials also has the effect of averaging each
  of the roots individually, and that it is possible to identify and
  exploit these situations.
The key to doing this systematically is the classical notion of {\em interlacing}.
\begin{definition}[Interlacing]  Let $f$ be a degree $n$ polynomial with real
roots $\{\alpha_i\}$, and let $g$ be degree $n$
or $n-1$ with real roots $\{\beta_i\}$ (ignoring $\beta_n$ in the degree $n-1$ case). We say
that $g$ interlaces $f$ if their roots alternate, i.e.,
$$ \beta_n\le\alpha_n\le \beta_{n-1}\le \ldots \beta_1\le \alpha_1,$$
and the largest root belongs to $f$.

If there is a single $g$ which interlaces a family
  of polynomials $f_1,\ldots,f_m$, we say that they have a {\em common
  interlacing}.
\end{definition}
It is an easy exercise to show that $f_1,\ldots,f_m$ of degree $n$ have a common interlacing iff there
  are closed intervals $I_n\le I_{n-1}\le\ldots I_1$ (where $\le$ means to the
  left of) such that the $i$th roots
  of all the $f_j$ are contained in $I_i$.
It is also easy to see that a set of polynomials has a common interlacing iff every
  pair of them has a common interlacing (this may be viewed as Helly's theorem
  on the real line).

We now state our main theorem about averages of polynomials with common
  interlacings.
\begin{theorem}[Lemma 4.1 in \cite{IF1}]\label{thm:interlace} Suppose $f_1,\ldots,f_m$ are
real-rooted of degree $n$ with positive leading
coefficients. Let $\lambda_k(f_j)$ denote the $k^{th}$ largest root of $f_j$ and
let $\mu$ be any distribution on $[m]$. If
$f_1,\ldots,f_m$ have a common interlacing, then for all $k=1,\ldots,n$
$$\min_j \lambda_k(f_j)\le \lambda_k(\E_{j\sim \mu}
f_j)\le \max_j \lambda_k(f_j).$$
\end{theorem}
The proof of this theorem is a three line exercise, which essentially amounts to
  applying the intermediate value theorem inside each interval $I_i$.

An important feature of common interlacings is that their existence is
  {\em equivalent} to certain real-rootedness statements.
Often, this characterization gives us a systematic way to argue that common
  interlacings exist.
The following seems to have been discovered a number of times.
It appears as Theorem~$2.1$ of Dedieu~\cite{Dedieu}, (essentially) as
Theorem~$2'$ of Fell~\cite{Fell}, and as (a special case of) Theorem 3.6 of Chudnovsky and
Seymour~\cite{ChudnovskySeymour}.
The proof of it included below assumes that the roots of a 
  polynomial are continuous functions of its coefficients (which may be 
  shown using elementary complex analysis).
\begin{theorem}\label{thm:fell} If $f_1,\ldots,f_m$ are degree $n$ polynomials and 
all of their convex combinations $\sum_{i=1}^m \mu_if_i$ have real roots,
then they have a common interlacing.
\end{theorem}
\begin{proof} Since common interlacing is a pairwise condition, it suffices to
handle the case of two polynomials $f_0$ and $f_1$. Let $$f_t := (1-t)f_0+tf_1$$
with $t\in [0,1]$.  Assume without loss of generality that $f_0$ and $f_1$ have
no common roots (if they do, divide them out and put them back in at the end). As $t$ varies from
$0$ to $1$, the roots of $f_t$ define $n$ continuous curves in the complex plane
$C_1,\ldots,C_n$, each beginning at a root of $f_0$ and ending at a root of $f_1$. 
By our assumption the curves must all lie in the real line. 
Observe that no curve can cross a root of either $f_0$ or $f_1$ in the middle:
if $f_t(r)=0$ for some $t \in (0,1)$ and $f_0(r)=0$, then immediately we also have
$f_t(r)=tf_1(r)=0$, contradicting the no common roots assumption. 
Thus, each curve defines a closed interval containing exactly one root of $f_0$
and one root of $f_1$, and these intervals do not overlap except possibly at
their endpoints, establishing the existence of a common interlacing.
\end{proof}

It is worth mentioning that the converse of Theorem~\ref{thm:fell} is true as
  well, but we will not use this fact.


While interlacing and real-rootedness are entirely univariate notions as
  discussed above, the most powerful ways to apply them arise by viewing them as
  restrictions of multivariate phenomena.
There are two important generalizations of real-rootedness to more than one
  variable: real stability and hyperbolicity.

We were inspired by the development of the theory of real stability
  in the works of Borcea and \Branden, including
  \cite{BBmixedDeterminants,BBWeylAlgebra,BBLeeYang1}.
Their results center primarily around characterizations of stable polynomials,
  including closure properties   (that is, operations that preserve real stability
  of polynomials)  and showing that properties of various mathematical structures
  an be related to the stability of some ``generating polynomial'' of that structure.

There is an isomorphism between real stable polynomials and 
  {\em hyperbolic polynomials}, a concept that originated in a series of papers by 
  G{\aa}rding \cite{gaarding}
 in his investigation of partial differential equations.
The theory of hyperbolic polynomials was developed further in the optimization community
  (see the survey of Renegar \cite{renegar}).
However, it was not until Gurvits's use of hyperbolic polynomials in his proof
  of the van der Waerden conjecture \cite{gurvitsOne}, that their combinatorial power
  was revealed.

While it is well known that the concepts of real stability and hyperbolicity are
essentially equivalent (one can translate easily between the two), various
features of the way each property is defined have led to a natural separation 
of results: algebraic closure properties and characterization in real stability
and analytic properties such as convexity in hyperbolicity. 
The ``method of interlacing polynomials" discussed in this survey, is in many ways a recipe for mixing the ideas 
  from these two communities into a single proof technique.

The method of interlacing polynomials consists of two somewhat distinct parts.  
The first is to show that a given collection of polynomials forms what we call
  an {\em interlacing family}, which is broadly speaking any class of polynomials
  for which the roots of its average can be related to those of the individual
  polynomials.
This falls naturally into the realm of results regarding real stable
  polynomials as it often reduces to that showing various linear
  combinations of polynomials are real-rooted.
The second part is to bound one of the roots of the expected polynomial under some distribution.
This is more of an analytic task, for which the convexity properties studied in
  the context of hyperbolicity are relevant.
For instance, in \cite{IF2}, the analysis of the largest root is based on
  understanding the evolution of the root surfaces defined by a 
  multivariate polynomial as certain differential operators are applied to it, and draws on the same convexity properties 
  that are at the core of hyperbolic polynomials.

\section{Restricted Invertibility}\label{sec:restricted}
The purpose of this section is to give the simplest possible demonstration of
  the method of interlacing families of polynomials.
It will be completely elementary and self-contained, relying only on classical facts about
  univariate polynomials, and should be accessible to an undergraduate.
Nonetheless, it is structurally almost identical to the proof of
  Weaver's conjecture and contains most of the same conceptual components in a
  primitive form.

Bourgain and Tzafriri's restricted invertibility theorem \cite{BT} states that any square
  matrix $B$ with unit length columns and small operator norm contains a large 
  column submatrix $B_S$ which is well-invertible on its span.
That is, the least singular value of the submatrix,
   $\sigma_{|S|}(B_S)$, is large.
This may be seen as a robust, quantitative version of the fact that any matrix
  contains an invertible submatrix of size equal to its rank.
The theorem was generalized to arbitrary rectangular $B$ by Vershynin
  \cite{versh}, and further sharpened in \cite{SpielmanSrivastava,youssef}.
We will give a proof of the following theorem from \cite{SpielmanSrivastava}, which corresponds to
  the important case $BB^T=I$, when the columns of $B$ are isotropic.

\begin{theorem}\label{thm:isotropic} Suppose $v_1,\ldots,v_m\in\C^n$ are vectors
with $\sum_{i=1}^mv_iv_i^T=I_n$. Then for every $k<n$ there is a subset $S\subset [m]$ of size $k$
with
$$ \lambda_k\left(\sum_{i\in S}v_iv_i^T\right)\geq
\left(1-\sqrt{\frac{k}{n}}\right)^2\frac{n}{m}.$$
\end{theorem}
The proof of this theorem has two
  parts. 
The first part is the special case of Theorem \ref{thm:comparison}
  in which $r_{1}, \dots , r_{n}$ are independent and identically distributed (i.i.d.)
  and $\expec{}{r_{i} r_{i}^{*}} = cI$.
It reduces the problem of showing the existence of a good subset to that of
  analyzing the roots of the expected characteristic polynomial.
\begin{lemma}\label{lem:induct}
Suppose $\v_1,\ldots,\v_k$ are i.i.d. copies of a finitely supported random vector $\v$ with $\E\v\v^*=cI$. Then, with positive probability,
$$ \lambda_k \left(\sum_{i=1}^k \v_{i}\v_{i}^*\right) \ge \lambda_k
\left(\E
\chi\left[\sum_{i=1}^k \v_i\v_i^*\right]\right).$$
\end{lemma}
The second part is the calculation of the expected polynomial and the derivation of a bound on its
  roots. 
\begin{lemma}\label{lem:lowerbarrier} Suppose $\v_1,\ldots,\v_k$ are i.i.d. copies of a random vector
  $\v$ with $\E\v\v^*=I$. Then,
  $$\E\chi\left[\sum_{i=1}^k
  \v_i\v_i^*\right](x)=(1-D)^kx^n=x^{n-k}(1-D)^nx^k.$$
Moreover,
$$ \lambda_k\left((1-D)^nx^k\right)\ge \left(1-\sqrt{\frac{k}{n}}\right)^2n.$$
\end{lemma}

\subsection{Interlacing and $(1-D)$ operators}
Let us begin with the first part. To relate the expected characteristic
  polynomial to its summands, we will inductively apply Theorem
  \ref{thm:interlace}, which requires the existence of certain common
  interlacings.
These will be established by a combination of two ingredients. 
The first is the following classical fact, which says that rank-one updates 
  naturally cause interlacing.
\begin{lemma}[Cauchy's Interlacing Theorem]\label{lem:cauchy} If $A$ is a symmetric matrix and $v$
is a vector then $\charp{A}{x}$ interlaces $\charp{A+vv^*}{x}$.
\end{lemma}

One can easily derive this from the
  {\em matrix determinant lemma}:
\begin{lemma}\label{lem:rank1update}
If $A$ is an invertible matrix and $u,v$ are vectors, then
\[
\mydet{A + uv^*} = \mydet{A} (1 + v^* A^{-1} u)
\]
\end{lemma}

The second ingredient is the following correspondence between 
  isotropic random rank one updates and differential operators.
\begin{lemma}\label{lem:iminusd} Suppose $\v$ is a random vector with $\E\v\v^*
= cI$ for some constant $c\ge 0$.
Then for every matrix $A$, we have
$$ \E \chi\left[A+\v\v^*\right](x) = (I-cD)\chi\left[A\right](x),$$
where $D$ denotes differentiation with respect to $x$.
\end{lemma}
\begin{proof} 
Using Lemma~\ref{lem:rank1update}, we obtain
\begin{align*}
\E \det(xI-A-\v\v^*) &= \E\det(xI-A)(1-\v^*(xI-A)^{-1}\v)
\\&= \det(xI-A)(1- \trace{(\E \v\v^*) (xI-A)^{-1}})
\\&= \det(xI-A)\left(1-c\tr(xI-A)^{-1}\right)
\end{align*}
Letting $\lambda_1,\ldots,\lambda_n$ denote the eigenvalues of $A$, this quantity becomes
$$ \prod_{i=1}^n(x-\lambda_i)\left(1-c\sum_{i=1}^n
\frac{1}{x-\lambda_i}\right) = \chi(A)(x)-c\sum_{i=1}^n\prod_{j\neq
i}(x-\lambda_j) = (1-cD)\chi(A) (x),$$
as desired.
\end{proof}
The purpose of Lemma \ref{lem:iminusd} is twofold.
First, it allows us to easily calculate expected characteristic
  polynomials, which a priori could be intractably complicated sums.
Second, the operators $(1-cD)$ have other nice properties
  which witness that the expected polynomials we generate
  have real roots and common interlacings.
\begin{lemma}[Properties of Differential Operators]\label{lem:iminusdproperties}
\
\begin{enumerate}
\item [(1)] If $f$ has real roots then so does $(I-cD)f$.
\item [(2)] If $f_1,\ldots,f_m$ have a common interlacing, then so do
$(I-cD)f_1,\ldots, (1-cD)f_m$.
\end{enumerate}
\end{lemma}
\begin{proof} 
For part (1), assume that $f$ and $f'$ have no common roots
  (otherwise, these are also common roots of $f$ and $f-cf'$ which are clearly
  real).
Consider the rational function
$$ \frac{f(x)-cf'(x)}{f(x)}=1-c\frac{f'(x)}{f(x)}=1-c\sum_{i=1}^n\frac{1}{x-\lambda_i}$$
  where $\lambda_i$ are the roots of $f$.
Inspecting the poles of this function and applying the intermediate value
  theorem shows that $f-cf'$ has the same number of zeros as $f$, all distinct
  from those of $f$.

For part (2), Theorem \ref{thm:interlace} tells us that all convex combinations
$\sum_{i=1}^m\mu_if_i$ have real roots. By part (1) it follows that all 
$$(1-cD)\sum_{i=1}^m\mu_if_i = \sum_{i=1}^m \mu_i (1-cD)f_i$$
also have real roots. By Theorem \ref{thm:fell}, this means that the $(1-cD)f_i$
must have a common interlacing.\end{proof}

With these facts in hand, we can easily complete the proof of Lemma
  \ref{lem:induct}.
\begin{proof} 
Assume $\v$ is uniformly distributed on some set $v_1,\ldots,v_m\in\C^n$. We need to show that
  there is a choice of indices $j_1,\ldots, j_k\in [m]$ for which 
$$ \lambda_k \left(\sum_{i=1}^k v_{j_i}v_{j_i}^*\right) \ge \lambda_k
\left(\E
\chi\left[\sum_{i=1}^k \v_i\v_i^*\right]\right).$$
For any partial assignment $j_1,\ldots,j_\ell$ of the indices,
 consider the ``conditional expectation'' polynomial:
$$ q_{j_1,\ldots,j_\ell}(x) := \E_{\v_{\ell+1},\ldots,\v_k} 
\chi\left[\sum_{i=1}^\ell v_{j_i}v_{j_i}^* +  \sum_{i=\ell+1}^k
\v_i\v_i^*\right].$$
Since the $\v_i$ are independent, and $\E {\v_{i}} = (1/m) I$, applying Lemma
  \ref{lem:iminusd} $k-\ell$ times reveals that:
  $$q_{j_1,\ldots,j_\ell}(x)=(1-(1/m)D)^{k-\ell}\chi\left[\sum_{i=1}^\ell
  v_{j_i}v_{j_i}^*\right](x).$$

We will show that there exists a $j_{\ell+1}\in [m]$ such that 
\begin{equation}\label{eqn:cinterlace}
\lambda_k(q_{j_1,\ldots,j_{\ell+1}})\ge
\lambda_k(q_{j_1,\ldots,j_\ell}),\end{equation}
which by induction will complete the proof.
Consider the matrix
$$A = \sum_{i=1}^\ell v_{j_i}v_{j_i}^*,$$
By Lemma \ref{lem:cauchy}, $\chi[A]$ interlaces
$\chi[A+v_{j_{\ell+1}}v_{j_{\ell+1}}^*]$ for
every $j_{\ell+1}\in [m]$. 
Lemma \ref{lem:iminusdproperties} tells us $(1-(1/m)D)$ operators preserve common
interlacing, so the polynomials 
$$ (1-(1/m)D)^{k-(\ell+1)}\chi(A+v_{j_{\ell+1}}v_{j_{\ell+1}}^*) =
q_{j_1,\ldots,j_\ell,j_{\ell+1}}(x)$$
must also have a common interlacing.
Thus, some $j_{\ell+1}\in [m]$ must satisfy \eqref{eqn:cinterlace}, as desired.
\end{proof}

\subsection{Laguerre Polynomials and the Univariate Barrier
Argument}\label{sec:unibarrier}
We now move on to the second part, Lemma \ref{lem:lowerbarrier}, in which we
  prove a bound on the $k$th root of the expected polynomial, which after
  rescaling by a factor of $m$ is just:
  $$\E \chi\left[m\cdot\sum_{i=1}^k  \v_i\v_i^*\right](x)= (1-D)^kx^n.$$
We begin by observing that $(1-D)^kx^n=x^{n-k}(1-D)^nx^k$. This may be verified by
term-by-term calculation, or by appealing to the correspondence between $(1-D)$
operators and random isotropic rank one updates established in Lemma
\ref{lem:iminusd} as follows.
Let $G$ be an $n$-by-$k$ matrix of random, independently distributed, $N (0,1)$
  entries.
The covariance matrix of each column is the $n$-dimensional identity matrix,
  and the covariance of each row is the $k$-dimensional identity.
So,
  \begin{align*}
  (1-D)^kx^n &= \E_{G} \chi(GG^*)(x)
  \\&= \E_G x^{n-k} \chi(G^*G)(x)
  \\&= x^{n-k}(1-D)^n x^k.
\end{align*}

Thus, we would like to lower bound the least root of $(1-D)^nx^k$.
The easiest way to do this is to observe that it is a constant multiple of a
  known polynomial, namely an
  {\em associated Laguerre polynomial} $\L_k^{(n-k)}(x)$.
These are classical orthogonal polynomials and a lot is known about the
  locations of their roots; in particular, they are known to be contained in the
  interval $[n(1-\sqrt{k/n})^2,n(1+\sqrt{k/n})^2]$ (see, for instance,
  \cite{krasikov}).

In order to keep the presentation self-contained, and also because
  it is a key tool in the proof of Kadison--Singer and more generally in the
  analysis of expected characteristic polynomials, we now give a direct
  proof of Lemma \ref{lem:lowerbarrier} based on the ``barrier method'' introduced
  in \cite{BSS}.
The basic idea is to study the effect of each $(1-D)$ operator on the roots of a
  polynomial $f$ via the associated rational function 
\begin{equation}\label{eqn:barrierdef}
\Phi_f(b):=-\frac{f'(b)}{f(b)}=-\frac{\partial \log f(b)}{\partial
b}=\sum_{i=1}^n\frac{1}{\lambda_i-b},\end{equation}
  which we will refer to as the {\em lower barrier function}.
The poles of this function are the roots $\lambda_1,\ldots,\lambda_n$ of $f$, and we remark that it is the
  same up to a multiplicative factor of $(-1/n)$ as the Stieltjes transform of the discrete measure supported on
  these roots.
It is immediate from the above expression that $\Phi_f(b)$ is
  positive, monotone increasing, and convex for $b$ is strictly less than
  the roots of $f$, and that it tends to infinity as $b$ approaches the smallest root of
  $f$ from below.

We now use the {\em inverse} of $\Phi_f$ to define a robust lower bound for the roots of a polynomial
$f$:
  $$\smin_\varphi(f):=\min\{x\in\R:\Phi_f(x)=\varphi\},$$
where $\varphi>0$ is a sensitivity parameter.
Since $\Phi_f(b)\rightarrow 0$ as $b\rightarrow -\infty$, it is immediate that
  we always have $\smin_\varphi(f)\le\lambda_{min}(f)$.
The number $\varphi$ controls the tradeoff between how accurate a lower bound 
  $\smin_\varphi$ is an how smoothly it varies --- in particular the extreme cases are
  $\smin_\infty(f)=\lambda_{min}(f)$, which is not always well-behaved, and 
  $\smin_0(f)=-\infty$, which doesn't even depend on $f$.
This quantity was implicitly introduced and used in \cite{BSS} and explicitly defined in
  \cite{sv}, where it was called the `soft spectral edge'; for an intuitive
  discussion of its behavior in terms of an electrical repulsion model, we
  refer the reader to the latter paper.

We also remark that the inverse Stieltjes transform was used by Voiculescu in
  his development of Free Probability theory to
  study the limiting spectral distributions of certain random matrix ensembles as the
  dimension tends to infinity.
We view the use of
  $\smin$ as a non-asymptotic analogue of that idea,
  except that we use it to reason about the edge of the
  spectrum rather than the bulk.

The following lemma tells us that $\smin_\varphi(f)$ grows in a smooth and predictable way
  when we apply a $(1-D)$ operator to $f$.
It is similar to Lemma 3.4 of \cite{BSS}, which was written in
  the language of random rank one updates of matrices.

\begin{lemma} \label{lem:lowershift} If $f$ has real roots and $\varphi>0$, then
  $$\smin_\varphi((1-D)f)\ge \smin_\varphi(f)+
  \frac{1}{1+\varphi}.$$
\end{lemma}
\begin{proof}
Let $b=\smin_\varphi(f)$. To prove the claim it suffices to find a $\delta\ge
(1+\varphi)^{-1}$ such that $b+\delta$ is below the roots of $f$ and 
$\Phi_{(1- D)f}(b+\delta)\le\varphi$.
We begin by writing the barrier function of $(1- D)$ in terms of the
  barrier function of $f$:
\begin{equation}\label{eqn:barrierupdate}
\Phi_{(1- D)f} = -\frac{(f- f')'}{f- f'} =
-\frac{(f(1+\Phi_f))'}{f(1+\Phi_f)} = -\frac{f'}{f}-\frac{
\Phi_f'}{1+\Phi_f}=\Phi_f-\frac{\Phi_f'}{1+\Phi_f}.
\end{equation}
This identity tells us that for any $\delta\ge 0$:
$$\Phi_{(1- D)f}(b+\delta) =
\Phi_f(b+\delta)-\frac{\Phi_f'(b+\delta)}{1+\Phi_f(b+\delta)},$$
which is at most $\varphi=\Phi_f(b)$ whenever
$$\frac{\Phi_f'(b+\delta)}{1+\Phi_f(b+\delta)} 
\ge \Phi_f(b+\delta)-\Phi_f(b).$$
This is in turn equivalent to
\[
\frac{
\Phi_f'(b+\delta)
}{
\Phi_f(b+\delta)-\Phi_f(b)
}
-
\Phi_f(b+\delta)
\geq 1.
\]
Expanding each $\Phi_f$ as a sum of terms as in \eqref{eqn:barrierdef} and applying Cauchy-Schwartz
  appropriately reveals\footnote{The simple but slightly cumbersome calculation appears as Claim 3.6 of
  \cite{BSS}; we have chosen to omit it here for the sake of brevity.}
  that the left-hand side of this inequality it at least
\[
   1 / \delta  - \Phi_f(b)
\]
This is at least $1$ for all $\delta\le (1+\varphi)^{-1}$.

We conclude that $\Phi_{(1- D)f}(b+\delta)$ is bounded by $\varphi$ for all $\delta\in
[0,(1+\varphi)^{-1}]$, which implies in particular that $b+\delta$ is
below the roots of $(1- D)f$.
 \end{proof}

Applying the lemma $n$ times immediately yields the following bound on
our polynomial of interest:
\begin{align*}
\lambda_k\left((1-D)^nx^k\right) 
&\ge\smin_\varphi\left((1-D)^nx^k\right) 
\\&\ge \smin_\varphi(x^k)+\frac{n}{1+\varphi}
\\&= -\frac{k}{\varphi}+\frac{n}{1+\varphi}
\qquad\textrm{since $\Phi_{x^k}(b)=-k/b$.}
\end{align*}

Setting $\varphi=\frac{\sqrt{k}}{\sqrt{n}-\sqrt{k}}$ yields Lemma \ref{lem:lowerbarrier},
  completing the proof of Theorem \ref{thm:isotropic}.
We remark that we have, as a byproduct, derived a sharp bound on the least
  root of an associated Laguerre polynomial.

In Lemma \ref{lem:smax} we use a multivariate version of the analogous bound for the largest
  root of the associated Laguerre polynomial.
A crucial aspect of the proof of the upper bound on the largest root is that
  it essentially depends only on the convexity and monotonicity of the barrier
  function.
For a real-rooted polynomial $f$, we define the {\em upper barrier function} as
  $\Phi^{f} (b) = f' (b) / f (b)$
and
\[
  \smaxx_\varphi(f):=\max\{x\in\R:\Phi^{f}(x)=\varphi\}.
\]

\begin{lemma}\label{lem:uppersmax}
 If $f$ has real roots and $\varphi>0$, then
  $$\smaxx_\varphi((1-D)f)\le \smaxx_\varphi(f)+
  \frac{1}{1-\varphi}.$$
\end{lemma}

\begin{proof}
Let $b = \smaxx_\varphi(f)$.
As before, we may derive 
\[
\Phi^{(1- D)f} 
=\Phi^{f}- (D \Phi^{f}) / ( 1-\Phi^{f}).
\]
So, to show that
\[
\smaxx_\varphi((1-D)f) \leq b + \delta ,
\]
it suffices to prove that
\[
\Phi^{f} (b) - \Phi^{f} (b+\delta)
\geq 
\frac{-D \Phi^{f} (b+\delta)}{1 - \Phi^{f} (b+d)}.
\]
As $\Phi^{f} (b)$ is monotone decreasing for $b$ above the roots of $f$,
  $D \Phi^{f} (b+\delta)$ is negative.
As $\Phi^{f} (b)$ is convex for the same $b$,
\[
\Phi^{f} (b) - \Phi^{f} (b+\delta)
\geq \delta (-D \Phi^{f} (b+\delta)).
\]
Thus, we only require
\[
  \delta \geq \frac{1}{1 - \Phi^{f} (b+d)}.
\]
As $\Phi^{f} (b)$ is monotone decreasing, this is satisfied
  for $\delta = 1 / (1 - \varphi)$.
\end{proof}

Setting $\varphi=\frac{\sqrt{k}}{\sqrt{n}+\sqrt{k}}$,
  we obtain our upper bound the largest root of an associated Laguerre polynomial.
\begin{lemma}\label{lem:upperLaguerre}
The largest root of  
  $(1-D)^n x^k$
 is at most $n(1+\sqrt{k/n})^2$.
\end{lemma}

\section{Mixed Characteristic Polynomials}\label{sec:mixed}
The argument given in the previous section is a special case of a more general principle: that
  the expected characteristic polynomials of certain random matrices can be expressed
  in terms of differential operators, which can then be used to establish the
  existence of common interlacings as well as to analyze the roots
  of the expected polynomials themselves.
In the isotropic case of Bourgain--Tzafriri, this entire chain of reasoning can
  be carried out by considering univariate polynomials only.
Morally, this is because the covariance matrices of all of the random vectors
  involved are multiples of the identity (which trivially commute with each other), 
  and all of the characteristic polynomials involved are simple univariate linear 
  transformations of each other (of type $(I-cD)$).

On the other hand, the proofs of Kadison-Singer and existence of Ramanujan graphs involve analyzing
  sums of independent rank one matrices which come from {\em non-identically
  distributed} distributions whose covariance matrices do not commute.
This leads to a much more general family of expected polynomials which we call
  {\em mixed characteristic polynomials}.
The special structure of these polynomials is revealed crisply when we view them
  as restrictions of certain multivariate polynomials.
Their qualitative and quantitative properties are, correspondingly,
  established using multivariate differential operators and barrier functions,
  which are analyzed using tools from the theory of real stable polynomials.

In the remainder of this section we will sketch a proof of Theorem \ref{thm:comparison}. 
The proof hinges on the following central identity, which describes the 
  general correspondence between sums of independent random rank one matrices
  and (multivariate) differential operators.
\begin{theorem}\label{thm:mixed}
Let $r_{1}, \dots , r_{m}$ be independent random column vectors in $\Complex{d}$.
For each $i$, let $A_{i} =  \expec{}{r_{i} r_{i}^{*}}$.
Then,
\begin{equation}\label{eqn:mixed1}
\expec{}{\charp{\sum_{i=1}^{m} r_{i} r_{i}^{*}}{x}}
= 
\left(\prod_{i=1}^{m} 1 - \partial_{z_{i}} \right) 
\mydet{x I + \sum_{i=1}^{m} z_{i} A_{i}}
\Big|_{z_{1} = \dots = z_{m} = 0}.
\end{equation}
\end{theorem}
In particular, the expected characteristic polynomial of a sum of independent
  rank one Hermitian random matrices is a function of the covariance matrices $A_i$. 
We call this polynomial the \textit{mixed characteristic polynomial} of
  $A_{1}, \dots , A_{m}$, and denote it by $\mixed{A_{1}, \dots , A_{m}}{x}$.
The name \textit{mixed characteristic polynomial} is inspired by the fact that
  the expected determinant of this matrix is called the mixed discriminant.
Notice that when $A_1=A_2=\ldots=A_m=I$, it is just a multiple
  of an associated Laguerre polynomial as in Section \ref{sec:restricted}.

Theorem \ref{thm:mixed} may be proved fairly easily by inductively applying an identity
  similar to Lemma \ref{lem:iminusd} or by appealing to the
  Cauchy-Binet formula; we refer the reader to \cite{IF2} for a short proof.
We remark that it and all of the other results in this section depend crucially 
  on the fact that the $r_ir_i^*$ are rank one, and fail rather spectacularly for 
  rank $2$ or higher matrices.

The most important consequence of Theorem \ref{thm:mixed} is that mixed
  characteristic polynomials always have real roots.
To prove this, we will need to consider a multivariate generalization of
  real-rootedness called real stability.
\begin{definition} A multivariate polynomial $f\in \R[z_1,\ldots,z_m]$ is {\em real stable}
  if it has no roots with all coordinates strictly in the upper half plane,
  i.e., if 
  $$ \imag{z_i}>0\quad\forall i\quad\Rightarrow f(z_1,\ldots,z_m)\neq 0.$$
\end{definition}
Notice that stability is the same thing as real rootedness in the univariate
  case, since complex roots occur in conjugate pairs.

A natural and relevant example of real stable polynomials is the
  following:
\begin{lemma}[\cite{BBmixedDeterminants}]\label{lem:det} If $A_1,\ldots,A_m$ are positive semidefinite
matrices, then
$$ f(z_1,\ldots,z_m)=\det\left(\sum_{i=1}^m z_iA_i\right)$$
is real stable.
\end{lemma}

One reason real stability is such a useful notion for us is that it has
  remarkable closure properties which are extremely well-understood
In particular, Borcea and \Branden \ have completely characterized the linear
  operators preserving real stability~\cite{BBWeylAlgebra}.
What this means heuristically is that proofs of stability can often be
  reduced to a formal exercise:  to prove that a particular polynomial is stable,
  one must simply write it as a composition of known stability-preserving
  operations. 

To prove that mixed characteristic polynomials are real stable, we will only
  require the following elementary closure properties.
\begin{lemma}[Closure Properties] If $f(z_1,\ldots,z_m)$ is real stable, then
so are 
$$(1-\partial_{z_i})f\qquad\textrm{for every $i$}$$
and 
$$f(\alpha,z_2,\ldots,z_m)\qquad\textrm{for every $\alpha\in\R$}.$$
\end{lemma}
The first part was essentially established by Lieb and Sokal in \cite{LiebSokal}.
It follows easily by considering a univariate restriction to $z_i$ and
  studying the associated rational function, as in the the (entirely univariate) proof of
  Lemma \ref{lem:iminusdproperties}.
The second part is trivial for $\alpha$ strictly in the upper half plane, and
  may be extended to the real line by appealing to Hurwitz's theorem.

Combining these properties with Theorem 4.1 instantly establishes the following important fact.
\begin{theorem}\label{thm:mixedreal} If $A_1,\ldots,A_m$ are positive
semidefinite, then $\mixed{A_1,\ldots,A_m}{x}$ is real-rooted.
\end{theorem}

We are now in a position to prove Theorem \ref{thm:comparison}. As in Lemma
  \ref{lem:induct}, we will do this inductively by showing 
  that the relevant ``conditional expectation'' polynomials have common
  interlacings.
However, instead of explicitly finding these common interlacings using Cauchy's
  theorem, we will guarantee their existence implicitly using Theorem
  \ref{thm:mixedreal}.
\begin{proof}[Proof of Theorem \ref{thm:comparison}]
For any partial assignment $v_1,\ldots,v_\ell$ of
  $r_1,\ldots,r_\ell$, consider the conditional expected polynomial
$$ q_{v_1,\ldots,v_\ell}(x):=\E \charp{\sum_{i=1}^\ell v_iv_i^* +
\sum_{i=\ell+1}^m r_ir_i^*}{x}.$$
Suppose $r_{\ell+1}$ is supported on $w_1,\ldots,w_N$. Then, for all convex
  coefficients $\sum_{i=1}^N \mu_i=1, \mu_i\ge 0$, the convex combination
  $$ \sum_{i=1}^N \mu_i q_{v_1,\ldots,v_\ell,w_i}(x)$$
is itself a mixed characteristic polynomial, namely
$$\mixed{v_1v_1^*,\ldots,v_\ell v_\ell^*, \sum_{i=1}^N \mu_i w_iw_i^*, \E
r_{\ell+2}r_{\ell+2}^*,\ldots,\E r_mr_m^*}{x},$$
which has real roots by Theorem \ref{thm:mixedreal}. This establishes
that the $q_{v_1,\ldots,v_\ell,w_i}(x)$ have a
  common interlacing, which by Theorem \ref{thm:interlace} implies that for
  every $k$ there exists an $i\in [N]$ for which 
  $$ \lambda_k\left(q_{v_1,\ldots,v_m,w_i}(x)\right) \le
  \lambda_k\left(q_{v_1,\ldots,v_m}(x)\right),$$
completing the induction.\end{proof}
The above proof highlights the added flexibility of allowing the $r_i$ to have
  different distributions: by taking some of these distributions to be
  deterministic, we can encode any conditioning and more generally any addition
  of a positive semidefinite matrix while remaining in the class of mixed
  characteristic polynomials.

\section{Analysis of Expected Polynomials}
In this section, we describe two situations in which we are able to bound
  the largest roots of mixed characteristic polynomials.
The first is very specific: we observe that the expected characteristic
  polynomial of a random signing of an adjacency matrix of a graph is
   equal, up to a shift, to the matching polynomial of the graph.
The zeros of this polynomial have been studied for decades and elementary
  combinatorial arguments due to Heilmann and Lieb \cite{heilmannLieb} can be used to give
  a sharp bound on its largest root.
The main consequence of this bound is the existence of infinite families of
  bipartite Ramanujan graphs of every degree.

The second situation is almost completely general. We show that given any
  collection of matrices satisfying $\sum_{i=1}^m A_i=I$, the mixed characteristic
  polynomial $\mixed{A_1,\ldots,A_m}{x}$ has roots bounded by $(1+\sqrt{\max_i
  \tr(A_i)})^2$.
This is achieved by a direct multivariate generalization of the barrier function
  argument that we used in Section \ref{sec:restricted} to upper bound the roots
  of associated Laguerre polynomials.
The main consequence of this bound is a proof of Weaver's conjecture and thereby
  a positive solution to the Kadison--Singer problem.\\

\subsection{Matching Polynomials}\label{sec:matchingpoly}
We are now ready to prove the bound \eqref{eqn:IF1} and thereby Theorem \ref{thm:IF1}.
For any $d-$regular graph $G=(V,E)$, let the random vectors
  $\{r_{(a,b)}\}_{(a,b)\in E}$ be defined as in \eqref{eqn:rabdef}.
Applying Theorem \ref{thm:comparison} with $k=1$ and subtracting $d$ from both
  sides, we find that:
\begin{align*}\lambda_{max}\left(\sum_{(a,b)\in E}
r_{(a,b)}r_{(a,b)}^*-dI\right)
&= \lambda_{max}\left(\sum_{(a,b)\in E} r_{(a,b)}r_{(a,b)}^*\right)-d
\\&\le
\lambda_{max}\left(\E\charp{\sum_{(a,b)\in E} r_{(a,b)}r_{(a,b)}^*}{x}\right)-d
\\&=\lambda_{max}\left(\E\charp{\sum_{(a,b)\in E}
r_{(a,b)}r_{(a,b)}^*-dI}{x}\right),\end{align*}
with positive probability.
Switching back to signed adjacency matrices by applying
  \eqref{eqn:signedlaplacian}, we conclude that 
\begin{equation}\label{eqn:goodsigning}\lambda_{max}(A_s)\le\lambda_{max}\left(\E\charp{A_s}{x}\right)\end{equation}
  with positive probability for a uniformly random signing $A_s$.

We now observe that this expected characteristic polynomial is equal to 
  the {\em matching polynomial} of the graph.
A matching is a graph in which every vertex has degree at most one.
The matching polynomial is a generating function which counts the number of
matchings that are subgraphs of a graph; for a graph on $n$ vertices, it is defined as
  $$\mu_G(x) := \sum_{i=0}^{\lfloor n/2\rfloor} (-1)^ix^{n-2i}m_i,$$
where $m_i$ is the number of subgraphs of $G$ with  $i$ edges that are matchings.

Godsil and Gutman \cite{GodsilGutman} showed that the matching polynomial of a graph is
  equal to the expected characteristic polynomial of a random signing of its
  adjacency matrix:
\begin{equation}\label{eqn:gg}  \E\charp{A_s}{x}=\mu_G(x).\end{equation}
This identity may be proved easily by expanding $\charp{A_s}{x}=\det(xI-A_s)$ as
  a sum of permutations and observing that the only terms that do not vanish
  are the permutations with all orbits of size two, which correspond to the matchings.

About a decade before this, Heilmann and Lieb \cite{heilmannLieb} studied the matching
  polynomial in the context of monomer-dimer systems in statistical physics.
In that paper, they showed that $\mu_G(x)$ always has all real roots (a fact
  which we have also just proved by writing it as a shift of a mixed
  characteristic polynomial), and that
\begin{equation}\label{eqn:hl} \lambda_{max}(\mu_G(x))\le
  2\sqrt{d-1}\end{equation}
  for a graph with maximum degree $d$.
They proved this bound by finding 
  certain simple combinatorial recurrences satisfied by $\mu_G(x)$, induced by edge and vertex
  deletions.
The appearance of the number $2\sqrt{d-1}$ is not a coincidence; Godsil
  \cite{godsil} later showed using similar recurrences that $\mu_G(x)$ divides the
  characteristic polynomial of a certain tree associated with $G$, which is an
  induced subgraph of the infinite $d-$regular tree.

Combining \eqref{eqn:goodsigning}, \eqref{eqn:gg}, and \eqref{eqn:hl} yields
  Theorem \ref{thm:IF1}. 
There is also a generalization of this theorem which proves the existence of
  ``irregular'' Ramanujan graphs, which were not previously known to exist; we
  refer the interested reader to \cite{IF1} for details.
\subsection{The Multivariate Barrier Argument} \label{sec:multibarrier}
The tight bound of $2\sqrt{d-1}$ obtained above relies heavily
  on the fact that the random vectors $r_{(a,b)}$ of interest come from a graph
  and have combinatorial structure.
Remarkably, it turns out that we can prove a bound that is almost as sharp
  by completely ignoring this structure and relying only on the much weaker 
  property that the $rr^*$ are rank one matrices of bounded trace.
This type of generic bound is precisely what one needs to control the roots of the quite
  general mixed characteristic polynomials which arise in the proof of Weaver's
  conjecture, and thereby prove Kadison--Singer.
\begin{theorem}\label{thm:mixedbound} Suppose $A_1,\ldots,A_m$ are positive semidefinite matrices with
$\sum_{i=1}^mA_i=I$ and $\tr(A_i)\le\epsilon$. Then, 
\begin{equation}\label{eqn:mixedbound}\lambda_{max}\left(\mixed{A_1,\ldots,A_m}{x}\right)\le
(1+\sqrt{\epsilon})^2.\end{equation}
\end{theorem}
At a high level, the proof of this theorem is very similar to that of Lemma
  \ref{lem:upperLaguerre}: we express $\mixed{A_1,\ldots,A_m}{x}$ as a product of
  differential operators applied to some nice initial polynomial, and show that each
  differential operator perturbs the roots in a predictable way.
The difference is that the differential operators and roots are
  now multivariate rather than univariate.

To deal with this issue, we begin by defining a notion of multivariate
  upper bound: we say that $b\in\R^m$ is {\em above} the roots of a real stable polynomial
  $f(z_1,\ldots,z_m)$ if $f(z)>0$ for all $z\ge b$ coordinate-wise.
It is best to think of an ``upper bound'' for the roots of $f$ as a set rather
  than as a single point --- the set of all points above the roots of $f$.

As we did in the univariate case, we soften this notion by studying certain
  rational functions associated with $f$ which interact naturally with the
  $(1-\partial_{z_j})$ operators we are interested in.
For each coordinate $j$, define the {\em multivariate barrier function}
$$ \Phi^f_j(z_1,\ldots,z_m) = \frac{\partial z_j f(z_1,\ldots,z_m)}{f(z_1,\ldots,z_m)},$$
and notice that 
$$\Phi^f_j(z_1,\ldots,z_m)=\sum_{i=1}^d \frac{1}{z_j-\lambda_i},$$
where $\lambda_1,\ldots,\lambda_d$ are the roots of the univariate restriction
  obtained by fixing all the coordinates other than $z_j$.

For a sensitivity parameter $\varphi<1$, we define a {\em $\varphi$-robust upper
  bound} on $f(z_1,\ldots,z_m)$ to be any point $b$ above the roots of $f$ with
  $\Phi^f_j(b)\le\varphi$ for all $j$.
We denote the set of all such robust upper bounds by $\smax_\varphi(f)$.
The following multivariate analogue of Lemma \ref{lem:uppersmax}  holds
  for $\smax$.
It says that applying an $(1-\partial_{z_j})$ operator simply
  moves the set of robust upper bounds in direction $j$ by a small amount.
\begin{lemma}\label{lem:smax} If $f(z_1,\ldots,z_m)$ is real stable and
$\varphi<1$, then
$$\smax_{\varphi}\left((1-\partial_{z_j})f\right)\supseteq
\smax_{\varphi}(f)+\frac{1}{1-\varphi}\cdot e_j,$$
where $e_j$ is the elementary basis vector in direction $j$.
\end{lemma}
The proof of this lemma is syntactically almost identical to that of Lemma
  \ref{lem:uppersmax}, except that it is less obvious
  that the barrier functions $\Phi^f_j$ are monotone and
  convex in the coordinate directions.
In \cite{IF2} we prove this by appealing to a powerful representation theorem
  of Helton and Vinnikov \cite{heltonvinnikov}, which says that bivariate
  restrictions of real stable polynomials can always be written as determinants
  of positive semidefinite matrices, which are easy to analyze.
Later, elementary proofs of this fact were given by James Renegar (using tools
  from the theory of hyperbolic polynomials \cite{BGLS}) and Terence Tao
  (using a combination of elementary calculus and complex analysis, along with
  Bezout's theorem).

With Lemma \ref{lem:smax} in hand, one can prove Theorem
  \ref{thm:mixedbound} by an induction similar to the one we used in Lemma
  \ref{lem:lowerbarrier}.
We refer the reader to \cite{IF2} for details.

Applying Theorems \ref{thm:comparison} and \ref{thm:mixedbound} to the random
  vectors defined in \eqref{eqn:weaverpartition} immediately yields Theorem
  \ref{thm:mss2main}.

\section{Ramanujan Graphs and  Weaver's Conjecture}
We conclude by showing how the generic bound derived above may be used to
  analyze the random signings that occur in the proof of Theorem \ref{thm:IF1}.
This turns out to be very instructive and is quite natural, since when $G=(V,E)$ is $d-$regular, 
  \eqref{eqn:signisotropic} tells us that
  $$ \E \sum_{(a,b)\in E} \frac{r_{(a,b)}r_{(a,b)}^*}{d}= I.$$
Thus, each vector has the same norm $\|r_{(a,b)}\|^2 = 2/d$, and applying
  Theorems \ref{thm:comparison} and \ref{thm:mixedbound} shows that
  $$\sum_{(a,b)\in E} {r_{(a,b)}r_{(a,b)}^*}\le
  d\left(1+\sqrt{\frac{2}{d}}\right)^2 = d+2+2\sqrt{2d}$$
  with positive probability.
This bound has asymptotically the same dependence on $d$ as the correct bound
  established using matching polynomials.
Moreover, it immediately proves that the dependence on $\epsilon$ in Theorem
  \ref{thm:mixedbound} cannot be improved: if it could, the above argument would
  imply the existence of signings with largest eigenvalue $o(\sqrt{d})$,
  contradicting the Alon--Boppana bound. 
Thus, the matrices arising in the study of Ramanujan graphs witness the
  sharpness of our bounds on mixed characteristic polynomials.
\footnote{We remark that the Alon--Boppana bound can also be used to show that the
  dependence on $\alpha$ in Theorem \ref{thm:mss2main} itself is tight by recursively
  applying it to the Laplacian of the complete graph. We point the reader to
  \cite{blogpost} or \cite{harveyOliver} for a complete argument.}


\end{document}